\documentclass[1 [leqno,11pt]{amsart}
\usepackage{amssymb, amsmath,amsmath,latexsym,amssymb,amsfonts,amsbsy, amsthm}

\numberwithin{equation}{section}

\allowdisplaybreaks

\let\al=\alpha

\let\f=\frac

\let\pa=\partial

\def\R{\mathbf R}

\def\cK{\mathcal K}

\newcommand{\beq}{\begin{equation}}
\newcommand{\eeq}{\end{equation}}
\newcommand{\ben}{\begin{eqnarray}}
\newcommand{\een}{\end{eqnarray}}
\newcommand{\beno}{\begin{eqnarray*}}
\newcommand{\eeno}{\end{eqnarray*}}

\newtheorem{theorem}{Theorem}[section]

\newtheorem{lemma}[theorem]{Lemma}
\newtheorem{proposition}[theorem]{Proposition}
\newtheorem{corol}[theorem]{Corollary}

\newtheorem{Theorem}{Theorem}[section]

\begin{document}

\title[Global regularity of the Steady Prandtl equation]
{Global $C^\infty$ regularity of the Steady Prandtl equation}

\author{Yue Wang}
\address{ School of Mathematical Sciences, Capital Normal University, Beijing 100048, China
}
\email{yuewang37@pku.edu.cn}

\author{Zhifei Zhang}
\address{School of Mathematical Sciences, Peking University, 100871, Beijing, P. R. China}
\email{zfzhang@math.pku.edu.cn}

\begin{abstract}
In this paper, we prove the global $C^\infty$ regularity of the Oleinik's solution
for the steady Prandtl equation with favorable pressure.

\end{abstract}

\date{\today}

\maketitle

\section{Introduction}

In this paper, we study the steady Prandtl equation:
\begin{equation}\label{eq:sPrandtl}
  \left\{
  \begin{aligned}
    &u\pa_x u +v\pa_{y}u-\pa_{y}^2u=-\frac{dp}{dx}(x),\quad x\ge 0,\, y\ge 0,\\
    &\pa_xu+\pa_y v=0,\\
    &u|_{y=0}=v|_{y=0}=0\quad\mbox{and}\quad \displaystyle\lim_{y\to+\infty} u(x,y)=U(x),
    \end{aligned}
  \right.
\end{equation}
where the outer flow $\big(U(x),p(x)\big)$ satisfies
\begin{align}\label{Bernoullilaw}
U(x)U'(x)+p'(x)=0.
\end{align}
This system derived by Prandtl could be used to describe the behavior of the solution near $y=0$ for the steady Navier-Stokes equations  when the viscosity coefficient is small.

The existence and regularity of  solution for  the system \eqref{eq:sPrandtl}  was proved by Oleinik \cite{Olei}  for a class of positive data $u_0(y)$ prescribed at $x=0$. Let us make it precise. We denote by $\cK$ a class of functions, which satisfy
\begin{align*}
 &u\in C_b^{2,\alpha}\big([0,+\infty)\big)(\al>0),\quad u(0)=0,\,\,u_y(0)>0,\,\,u(y)>0\,\, \text{for}\, \,y\in(0,+\infty),\\
 &\displaystyle\lim_{y\to+\infty} u(y)=U(0)>0,\quad u_{yy}(y)-\frac{dp}{dx}(0)=O(y^2).
\end{align*}
Oleinik proved the following classical result(see Theorem 2.1.1 in \cite{Olei}).

\begin{Theorem}\label{thm:olei}
If $u_0\in \cK$ and $\frac{dp}{dx}(x)$ is smooth, then there exists $X>0$ such that the steady Prandtl equation \eqref{eq:sPrandtl} admits a solution $u\in C^1([0,X)\times\R_+)$ with the following properties:

\begin{itemize}
\item[1.] regularity: $u$ is bounded and continuous in $[0,X]\times \R_+; \,u_{y},u_{yy}$ are bounded and continuous in $[0,X)\times\R_+$; and $v,v_{y},u_{x}$ are locally bounded and continuous in $[0,X)\times\R_+.$
\item[2.] non-degeneracy: $u(x,y)>0$ in $[0,X)\times (0,+\infty)$ and for all $\bar{x}<X,$ there exists $y_0>0,m>0$ so that
\beno
\partial_yu(x,y)\geq m\quad \text{in}\,\, [0,\bar{x}]\times[0,y_0].
\eeno
\item[3.] global existence: if $p'(x)\le 0$, then the solution is global in $x$.\end{itemize}
\end{Theorem}

Then the following problems are natural and important:

\begin{itemize}

\item {\bf Problem 1.} Boundary layer separation in the case of unfavorable pressure, i.e., $p'(x)>0$.

\item{\bf Problem 2.}  Global $C^\infty$ regularity of Oleinik's solution in the case of favorable pressure, i.e., $p'(x)\le 0$.

\item{\bf Problem 3.} Vanishing viscosity limit of the steady Navier-Stokes equations.

\end{itemize}

For Problem 1,  Dalibard and Masmoudi \cite{DM} proved  the boundary layer separation for a class of special data and $p'(x)=1$, and show that the solution behaves near the separation:
\beno
\pa_yu(x,0)\sim (x^*-x)^\f12,\quad x<x^*.
\eeno
Shen, Wang and Zhang \cite{SWZ} proved the boundary layer separation for the Oleinik's type data in $\cK$ and $p'(x)>0$, and
studied the local behavior of the solution near the separation point. See also \cite{E} for an unpublished result.

For Problem 3, there are some important progress on the stability  for some special boundary layer flows such as the Blasius flow and shear flow \cite{Guo1, GM}.

Global $C^\infty$ regularity of Oleinik's solution is also a long-standing problem. The main challenge is to study the regularity of the solution for a degenerate parabolic equation. Indeed, Oleinik's proof is based on the Von Mises transformation:
\begin{align}\label{eq:VM}
 \psi(x,y)=\int_0^yu(x,z)dz,\quad w=u^2.
\end{align}
A direct calculation shows that
  \begin{align}\label{PvM}
  \begin{split}
& 2\partial_y u =\partial_\psi w,\quad2\partial^2_{y} u=\sqrt{w}\partial^2_{\psi}w.
 \end{split}
\end{align}
Then the new unknown $w(x,\psi)$ satisfies
  \begin{align}\label{eq:sPrandtl-VM}
 \pa_xw-\sqrt{w}\partial^2_{\psi}w=-2\frac{dp}{dx}\quad\text{in}\quad[0,X)\times \R_+,
 \end{align}
together with
\ben\label{eq:sPrandtl-bc}
w(0,\psi)=w_0(\psi)=u_0(y)^2,\quad w(x,0)=0,\quad \displaystyle\lim_{\psi\to+\infty}w(x,\psi)=U(x)^2.
\een
Now  \eqref{eq:sPrandtl-VM} is a degenerate parabolic equation due to $w=0$ at $\psi=0$. 

Recently, Guo and  Iyer \cite{Guo2} proved the higher regularity of the  solution in a local time, i.e., $0<x<X\ll1$.
The goal of this paper is to give an affirmative answer to {\bf Problem 2.} The following is our main result.

\begin{Theorem}\label{thm:reg}
Let $u$ be a global solution to \eqref{eq:sPrandtl} constructed in Theorem \ref{thm:olei} with $u_0\in \cK$ and $\frac{d p(x)}{dx}\leq 0$ smooth. For any positive integers $m,k$ and any positive constants $X, Y$ with $\epsilon<X$,
 there exists a positive constant $C$ depending only on $\epsilon, X, Y, u_0, p, k, m$ so that
 \beno
|\pa_x^{k}\pa_y^{m} u(x,y)|\le C\quad \text{for}\quad (x,y)\in [\epsilon, X]\times[0, Y].
\eeno
\end{Theorem}

For $y\in [\delta, +\infty)$ and $\delta>0$, \eqref{eq:sPrandtl-VM} is an uniform parabolic equation due to $w(x,y)>0$ for $y\ge \delta$ and \eqref{Bernoullilaw} with favorable pressure and $U(0)>0$.
Thus, in the domain $[\epsilon, +\infty)\times[\delta, +\infty)$, the $C^\infty$ regularity of the solution is a direct consequence of classical parabolic regularity theory \cite{Kry}. Therefore, we will focus on the domain $[\epsilon, X]\times[0, \delta]$ for some small $\delta>0$.

Our result does not give the uniform bounds of $|\pa_x^k\pa_y^mu(x,y)|$ in $x,y$. In particular,  the large time behavior of the solution is a very interesting problem(see also \cite{Iyer}),  which will be left to our future work. In the unsteady case, Xin and Zhang \cite{XZ} proved the global existence of weak solution under the favorable pressure. However, the global $C^\infty$ regularity up to the boundary 
remains open.

\section{Lower order regularity estimates}

In the sequel, we assume that $u$ is a global solution to \eqref{eq:sPrandtl} constructed in Theorem \ref{thm:olei} with $u_0\in \cK$ and $\frac{d p(x)}{dx}\leq 0$ smooth.

First of all, we improve the growth estimate of $|\pa_x w|.$

\begin{lemma}\label{lem:growth}
Let $0<\epsilon<X$, there exists $\delta_{1}>0$  so that
$$|\pa_x w(x,\psi)|<C\psi\quad \text{for} \quad (x,\psi)\in [\epsilon,X]\times[0,\delta_1].$$
Here $C$ is a constant independent of $\psi$.

\end{lemma}
\begin{proof}
Thanks to Theorem 2.1.14 in \cite{Olei}, it holds in $ [0, X]\times [0,\delta_1]$ that
\begin{align}\label{xd}\begin{split}
&|\partial_x w|\leq C\psi^{\frac{1}{2}+\alpha},\quad 0<m<\pa_{\psi} w<M\Rightarrow  m\psi< w<M\psi,\end{split}
\end{align}
for some $\alpha\in(0,\frac{1}{2}),$ and positive constants $m, M$ depending only on $X,u_0,p$.
Without loss of generality, we may assume $\delta_1<1.$

Take a smooth cutoff function $0\leq\zeta(x)\leq 1$ so that $\zeta(x)\equiv 1$ in $[\epsilon, X]$ and $\zeta(x)\equiv 0$ in $[0,\frac{\epsilon}{2}]$. Since $w$ is smooth in the interior, we have, in $(0,X]\times (0,\delta_1],$
 \begin{align}\nonumber
    \pa_x \pa_x w-\sqrt{w}\pa_\psi^2 \pa_x w=\frac{\pa_x w}{2\sqrt{w}}\pa_\psi^2 w-2\frac{d^2p}{dx^2}=\frac{(\pa_x w)^2}{2w}+2\frac{\pa_x w\frac{dp}{dx}}{2w}-2\frac{d^2p}{dx^2},
 \end{align}
which gives
 \begin{align*}
 \begin{split}
    \pa_x\big[\pa_x w\zeta(x)\big]-\sqrt{w}\pa_\psi^2\big[\pa_x w\zeta(x)\big]=&\frac{(\pa_x w)^2}{2w}\zeta(x)+2\frac{\pa_x w\frac{dp}{dx}}{2w}\zeta(x)-2\frac{d^2p}{dx^2}\zeta(x)\\&+\pa_x w\pa_x\zeta(x):=I.\end{split}
 \end{align*}
 From \eqref{xd}, we know that
 \begin{align}\label{I}
|I|\leq C\psi^{2\alpha}+C\psi^{\alpha-\frac{1}{2}}+C+C\psi^{\frac{1}{2}+\alpha}
 \leq C\psi^{\alpha-\frac{1}{2}}.
 \end{align}

Next we  take $\varphi(\psi)=A_1\psi-A_2\psi^{1+\beta}$ with constants $A_1,  A_2$ big and $\beta\in(0,1)$ to be determined.
Then by \eqref{xd} and \eqref{I}, we get
  \begin{align*}
  \begin{split}
    \pa_x [\pa_x w\zeta(x)-\varphi]-\sqrt{w}\pa_\psi^2 [\pa_x w\zeta(x)-\varphi]\leq&|I|-A_2\sqrt{w}\beta(1+\beta)\psi^{\beta-1}\\
    \leq &C\psi^{\alpha-\frac{1}{2}}-A_2\sqrt{m}\beta(1+\beta)\psi^{\beta-\frac{1}{2}}.
    \end{split}
 \end{align*}
 Taking $\beta={\alpha}$ and  $A_2$ large enough, we conclude
 $$
 \pa_x\big[\pa_x w\zeta(x)-\varphi\big]-\sqrt{w}\pa_\psi^2\big[\pa_x w\zeta(x)-\varphi\big]<0\quad in \quad (0,X]\times (0,\delta_1).
 $$
 We get by \eqref{xd} that for $x\le X$,
\beno
&&(\pa_x w\zeta-\varphi)(x,0)=0,\\
&&(\pa_x w\zeta-\varphi)(x,\delta_1)\leq M\delta_1^{\frac{1}{2}+\alpha}-A_1\delta_1+A_2\delta_1^{1+\beta},
\eeno
Requiring $A_1\geq A_2,$  by the definition of $\zeta(x),$ we have
$$(\pa_x w\zeta-\varphi)(0,\psi)\leq 0\quad\text{in}\quad [0,\delta_1],$$
and take $A_1$ large enough depending on $M,\delta_1, A_2$ so that
$$(\pa_x w\zeta-\varphi)(x,\delta_1)\leq 0.$$
Then the maximum principle ensures that
$$(\pa_x w\zeta-\varphi)(x,\psi)\leq 0\quad \text{in}\quad [0,X]\times [0,\delta_1],$$
which implies
$$\pa_x w(x,\psi)\leq  A_1\psi-A_2\psi^{1+\beta}\le \frac{ A_1}{2}\psi \quad \text{in}\quad [\epsilon,X]\times [0,\delta_1],$$
if $\delta_1$ is chosen suitably small.

The fact that  $-\pa_x w\leq \frac{ A_1}{2}\psi$ in $[\epsilon, X]\times [0,\delta_1]$  could be similarly proved by considering $-\pa_x w\zeta-\varphi$.
\end{proof}

\begin{proposition}\label{prop:lower}
Let $0<\epsilon<X$, there exists $\delta_{2}>0$  so that for $(x,\psi)\in [\epsilon,X]\times (0,\delta_2],$
\begin{align*}
|\pa_\psi\pa_xw(x,\psi)|\le C, \quad |\pa_x^2w(x,\psi)|\le C\psi^{-\f12},\quad
|\pa_\psi^2\pa_xw (x,\psi)|\leq C\psi^{-1}.
\end{align*}
Here $C$ is a constant independent of $\psi$.
\end{proposition}

\begin{proof} By Lemma \ref{lem:growth},  there exists $\delta_1>0$ so that
\begin{align}\label{wxQ}
|\pa_x w(x,\psi)|\le C\psi\quad \text{in} \quad \big[\frac{\epsilon}{2},X]\times[0,\delta_1].
\end{align}
For any $(x_3,\psi_3)\in[\epsilon,X]\times (0,\delta_0]$ where $\delta_0=\min\big\{\frac{2}{3}\delta_1,\frac{\epsilon}{2}\big\}$, we denote
\beno
Q=\big\{x_3-\psi_3^{\frac{3}{2}}\leq x\leq x_3\big\}\times \big\{\frac{1}{2}\psi_3\leq \psi\leq\frac{3}{2}\psi_3\big\}.
\eeno
By \eqref{wxQ} and the definition of $\delta_0,$ we have
\begin{align}\label{wxQ1}
|\pa_x w|\le C\psi\quad \text{in} \quad Q.
\end{align}

Now we make a transformation
\begin{align*}
 T: \quad&Q\longrightarrow[-1,0]_{\tilde{x}}\times\big[-\frac{1}{2},\frac{1}{2}\big]_{\tilde{\psi}}:=\widetilde{Q},
 \\&(x,\psi)\mapsto (\tilde{x},\tilde{\psi}),
\end{align*}
where
\begin{align*}
    x-x_3&=\psi_3^{\frac{3}{2}}\tilde{x},\quad \psi-\psi_3=\psi_3\tilde{\psi}.
\end{align*}

Thanks to $\pa_x=\frac{1}{\psi_3^{\frac{3}{2}}}\pa_{\tilde{x}}, \pa_\psi=\frac{1}{\psi_3}\pa_{\tilde{\psi}}$, we find that
\begin{align*}
\pa_{\tilde{x}}(\psi_3^{-1} w)-\frac{\sqrt{w}}{\psi_3^{\frac{1}{2}}}\pa_{\tilde{\psi}}^2( \psi_3^{-1} w)=-2\frac{d}{d\tilde{x}}p\,\psi_3^{-1}
\quad \text{in}\quad \widetilde{Q}.
\end{align*}
It follows from  \eqref{xd} that
\begin{align*}
0<c\leq \frac{\sqrt{w}}{\psi_3^{\frac{1}{2}}}\leq C, \quad |\psi_3^{-1} w|\leq C \quad \text{in} \quad \widetilde{Q},
\end{align*}
and  for any $z_1,z_2\in Q$\big(i.e., $\tilde{z}_1=Tz_1,\tilde{z}_2=Tz_2\in \widetilde{Q}\big)$,
\begin{align*}
    \left|\frac{\sqrt{w}}{\psi_3^{\frac{1}{2}}}(\tilde{z}_1)-\frac{\sqrt{w}}{\psi_3^{\frac{1}{2}}}(\tilde{z}_2)
    \right|&=\frac{1}{\psi_3^{\frac{1}{2}}}\frac{|w(z_1)-w(z_2)|}{\sqrt{w}(z_1)+\sqrt{w}(z_2)}
\\&\leq C \frac{\psi_3|\tilde{z}_1-\tilde{z}_2|}{\psi_3}=C|\tilde{z}_1-\tilde{z}_2|,
\end{align*}
which implies
\begin{align*}
\Big |\frac{\sqrt{w}}{\psi_3^{\frac{1}{2}}}\Big|_ {\mathcal{C}^\alpha(\widetilde{Q})}\leq C
\end{align*}
for any $\alpha\in (0,1)$. On the other hand, since $p$ is smooth, we have
$$
\Big|\frac{d}{d\tilde{x}}p\,\psi_3^{-1} \Big|_{\mathcal{C}^{0,1}([-1,0]_{\tilde{x}})}\leq C.
$$
By standard interior a priori estimates(see Theorem 8.11.1 in \cite{Kry}), we have
\begin{align}\label{psipsialpha}
    |\pa_{\tilde{\psi}}^2w \psi_3^{-1} |_ {\mathcal{C}^\alpha\big([-\frac{1}{2},0]_{\tilde{x}}\times[-\frac{1}{4},\frac{1}{4}]_{\tilde{\psi}}\big)}\leq C.
\end{align}

Let  $f=\pa_{x}w\psi_3^{-1}$, which satisfies
\beno
\pa_{\tilde{x}}f-\frac{\sqrt{w}}{\psi_3^{\frac{1}{2}}}\pa_{\tilde{\psi}}^2f-
\frac{\pa_{\tilde{\psi}}^2w}{2\sqrt{w}\psi_3^{\frac{1}{2}}}f=-2\frac{d}{d\tilde{x}}\frac{d}{dx}p\,\psi_3^{-1}.
\eeno
By \eqref{wxQ1}, we have
\begin{align*}
|f|\leq C \quad \text{in} \quad \widetilde{Q}.
\end{align*}
We write
\begin{align*}
\frac{\pa_{\tilde{\psi}}^2w}{2\sqrt{w}\psi_3^{\frac{1}{2}}}=\pa_{\tilde{\psi}}^2w \psi_3^{-1}\frac{\psi_3^{\frac{1}{2}}}{2\sqrt{w}}.
\end{align*}
Thanks to
\begin{align*}
\Big|\frac{\psi_3^{\frac{1}{2}}}{\sqrt{w}}(\tilde{z}_1)-\frac{\psi_3^{\frac{1}{2}}}{\sqrt{w}}(\tilde{z}_2)
    \Big|=\psi_3^{\frac{1}{2}}\frac{\Big|\frac{w(z_1)-w(z_2)}{w(z_1)w(z_2)}\Big|}{\frac{1}{\sqrt{w}}(z_1)+\frac{1}{\sqrt{w}}(z_2)}
\leq C|\tilde{z}_1-\tilde{z}_2|,
\end{align*}
we have
 \begin{align}\label{psichusqrtw}
 \Big |\frac{\psi_3^{\frac{1}{2}}}{\sqrt{w}}\Big|_ {\mathcal{C}^\alpha(\widetilde{Q})}\leq C,
\end{align}
which along with \eqref{psipsialpha} gives
$$
\Big|\frac{\pa_{\tilde{\psi}}^2w}{2\sqrt{w}\psi_3^{\frac{1}{2}}}\Big|_ {\mathcal{C}^\alpha([-\frac{1}{2},0]_{\tilde{x}}\times[-\frac{1}{4},\frac{1}{4}]_{\tilde{\psi}})}\leq C.$$
As before, since $\frac{dp}{dx}$ is smooth, we have
$$\Big|\frac{d}{d\tilde{x}}\frac{d}{dx}p\,\psi_3^{-1}\Big|_{\mathcal{C}^{0,1}([-1,0]_{\tilde{x}})}\leq C.$$
Then standard interior a priori estimates yield that
\beno
|\pa_{\tilde{x}}f|_ {{L}^\infty([-\frac{1}{4},0]_{\tilde{x}}\times[-\frac{1}{8},\frac{1}{8}]_{\tilde{\psi}})}+|\pa_{\tilde{\psi}}f|_ {{L}^\infty([-\frac{1}{4},0]_{\tilde{x}}\times[-\frac{1}{8},\frac{1}{8}]_{\tilde{\psi}})}+|\pa_{\tilde{\psi}}^2f|_ {{L}^\infty([-\frac{1}{4},0]_{\tilde{x}}\times[-\frac{1}{8},\frac{1}{8}]_{\tilde{\psi}})}\leq C.
\eeno
Especially,  it holds that
\beno
|\pa_x^2 w(x_3,\psi_3)|\leq C\psi_3^{-\frac{1}{2}},\quad|\pa_{\psi}\pa_xw(x_3,\psi_3)|\leq C,\quad|\pa_{\psi}^2\pa_xw(x_3,\psi_3)|\leq C\psi_3^{-1},
\eeno
which give our results.
\end{proof}

\section{Higher order regularity estimates}

In this section, we study the higher order regularity of the solution constructed in Theorem \ref{thm:olei} with $u_0\in \cK$ and $\frac{d p(x)}{dx}\leq 0$ smooth. The following is one of our main results.

\begin{proposition}\label{prop:high}
Let $0<\epsilon<X$ and $k\ge 2$, there exists $\delta>0$  so that in $[\epsilon,X]\times[0,\delta],$
\begin{align*}
|\pa_x^{k}w|\leq C\psi,\quad |\pa_\psi\pa_x^{k} w|\leq C, \quad |\pa_\psi^2\pa_x^{k} w|\le C\psi^{-1}.
\end{align*}
Here $C$ is a constant independent of $\psi$.
\end{proposition}

\begin{proof}
The proof is based on the induction argument. Thanks to Lemma \ref{lem:growth} and Proposition \ref{prop:lower},  we may inductively assume that for  $0\leq j\leq k-1,$ there holds that in $\big[\f \epsilon 2, X]\times [0, \delta_3]$,
\ben\label{eq:induct0}
|\pa_\psi\pa_x^{j} w|\le C,\,\,|\pa_\psi^2\pa_x^{j} w|\le C\psi^{-1},\,\,|\pa_x^{j} w|\le C\psi,\,\,|\pa_x^j\sqrt{ w}|\leq C\psi^{\frac{1}{2}},\,\,|\pa_x^k w|\leq C\psi^{-\frac{1}{2}}.
\een Without loss of generality, assume $\delta_3<<1.$
Our goal is to show that there exists $\delta_4<\delta_3$ so that in $[\epsilon,X]\times[0,\delta_4],$
\ben\label{eq:induct}
|\pa_\psi\pa_x^{k} w|\le C,\,|\pa_\psi^2\pa_x^{k} w|\le C\psi^{-1},\,|\pa_x^k w|\le C\psi,\,|\pa_x^k\sqrt{ w}|\leq C\psi^{\frac{1}{2}},\,|\pa_x^{k+1}w|\leq C\psi^{-\frac{1}{2}}.
\een
This will follow from the following Lemma \ref{lem:s1}, Lemma \ref{lem:s2} and Lemma \ref{lem:s3}.
\end{proof}

\begin{lemma}\label{lem:s1} 
Assume that \eqref{eq:induct0} holds. Then
it holds that in $[\frac{7\epsilon}{8},X]\times[0,\delta_3],$
\begin{align*}
|\pa_x^k w|<M_1\psi^{1-\beta},\quad |\pa_x^k\sqrt{ w}|\leq M_1\psi^{\frac{1}{2}-\beta}
\end{align*}
for any $0<\beta\ll1$.
\end{lemma}

\begin{proof}
Fix  any $h<\frac{\epsilon}{8}$. Set $D=(0,X]\times (0,\delta_3)$ and let
$$ g=\left\{
\begin{aligned}
&\frac{\pa_x^{k-1}w(x-h,\psi)-\pa_x^{k-1}w(x,\psi)}{-h}\zeta+M\psi\ln \psi\,\,\,&\frac{5\epsilon}{8}\leq x\leq X,\psi\geq 0,\\
&M\psi\ln \psi\,\,\,&0\leq x<\frac{5\epsilon}{8},\psi\geq0,
\end{aligned}
\right.
$$
where $M$ is a big positive constant to be determined, and $\zeta(x)$ is a smooth cutoff function so that
 $0\leq\zeta(x)\leq 1$, $\zeta(x)\equiv 1$ in $[\frac{7}{8}\epsilon,X]$, $\zeta(x)\equiv 0$ in $[0,\frac{5\epsilon}{8}]$.

 By the  definition of $g$ and continuity of $\pa_x^{k-1} w$ up to $\{\psi=0\},$ we know that $g$ is smooth in $D$ and continuous in $\bar{D}$(Growth estimate implies continuity, as pointed out on P38 in \cite{Olei}). By the assumption, we have $g(x,0)=0.$
By the definition of $\zeta,$ we have $g(0,\psi)\leq 0.$ Taking $M$ large enough, we have
for $x\in [0,X],$
\beno
g(x,\delta_3)\leq C(\delta_3)^{-\frac{1}{2}}+M\delta_3\ln\delta_3\leq 0.
\eeno

Next we show that by a proper choice of $M,$
 the positive maximum of $g$ cannot be achieved in the interior. We argue by the contradiction. Assume that there exists a point $(x_0,\psi_0)=p_0\in D$ such that
 $$g(p_0)=\max_{\overline{D}} g>0,$$
 which implies $x_0>\frac{5\epsilon}{8}$ and
 \begin{align}\label{sign}
  \frac{1}{-h}\big(\pa_x^{k-1} w(x_0-h,\psi_0)-\pa_x^{k-1} w(x_0,\psi_0)\big)> 0.
 \end{align}
 By \eqref{xd}, we have
\begin{align}\label{gpsi2}
  -\sqrt{w}\pa_\psi^2(M\psi\ln \psi)=-M\sqrt{w}\psi^{-1}\leq -\gamma M\psi^{-\frac{1}{2}}
\end{align}
for $\gamma=\sqrt{m}.$
By \eqref{eq:sPrandtl-VM}, we have in $D$,
\begin{align*}
  \pa_x \pa_x^{k-1} w-\sqrt{w}\pa_\psi^2 \pa_x^{k-1} w=&-2\frac{d^kp}{dx^k}+\sum_{m=1}^{k-2}C_{k-1}^m(\pa_x^{k-1-m}\sqrt{w})\pa_\psi^2 \pa_x^{m} w+(\pa_x^{k-1}\sqrt{w})\pa_\psi^2  w\\
  =&-2\frac{d^kp}{dx^k}+\sum_{m=1}^{k-2}C_{k-1}^m(\pa_x^{k-1-m}\sqrt{w})\pa_\psi^2 \pa_x^{m} w+\frac{\pa_x^{k-1}w}{2\sqrt{w}}\frac{\pa_x w}{\sqrt{w}}\\&+\Big(\frac{\pa_x^{k-1}w}{2\sqrt{w}}\Big)\frac{2}{\sqrt{w}}\frac{dp}{dx}+\sum_{m=0}^{k-3}C_{k-2}^m\pa_\psi^2w\pa_x^{m+1}w\pa_x^{k-2-m}\frac{1}{2\sqrt{w}}.
 \end{align*}
We denote
 \begin{align*}
    I_1&=-2\frac{d^kp}{dx^k}+\sum_{m=1}^{k-2}C_{k-1}^m(\pa_x^{k-1-m}\sqrt{w})\pa_\psi^2 \pa_x^{m} w+\frac{\pa_x^{k-1}w\pa_x w}{2w},\\
     I_2&=\Big(\frac{\pa_x^{k-1}w}{2\sqrt{w}}\Big)\frac{2}{\sqrt{w}}\frac{dp}{dx},\\
      I_3&=\sum_{m=0}^{k-3}C_{k-2}^m\pa_\psi^2w\pa_x^{m+1}w\pa_x^{k-2-m}\frac{1}{2\sqrt{w}}.
 \end{align*}
 Then for $x\geq\frac{5\epsilon}{8},$
 \begin{align}\label{eq:g1}
 \begin{split}
   \pa_x g_1-\sqrt{w(p_h)}\pa_\psi^2g_1=&\frac{\sqrt{w(p_h)}-\sqrt{w(p)}}{-h}\pa_\psi^2 \pa_x^{k-1} w(p)+\sum_{i=1}^3\frac{1}{-h}(I_i(p_h)-I_i(p)).
   \end{split}
 \end{align}
 where $g_1=\frac{1}{-h}\big(\pa_x^{k-1}w(p_h)-\pa_x^{k-1}w(p)\big)$ with $p_h=(x-h,\psi),\,\,p=(x,\psi).$

 Using the induction assumption \eqref{eq:induct0}, it is easy to verify that for $x\geq\frac{5\epsilon}{8},$
 \begin{align}
 \label{I0}
&\Big|\frac{1}{-h}(\sqrt{w}(p_h)-\sqrt{w}(p))\pa_\psi^2 \pa_x^{k-1} w(p)\Big|\leq C\psi^{-\frac{1}{2}},\\
 \label{I1}
   &\Big |\frac{1}{-h}(I_1(p_h)-I_1(p))\Big|\leq C\psi^{-\frac{1}{2}},\\
    \label{I3}
    &\Big|\frac{1}{-h}(I_3(p_h)-I_3(p))\Big|\leq C\psi^{-\frac{1}{2}}.
    \end{align}
We have
 \begin{align*}
   \frac{1}{-h}(I_2(p_h)-I_2(p))=g_1\cdot \Big[\frac{1}{w}\frac{dp}{dx}(p_h)\Big]+\pa_x^{k-1} w(p) \frac{1}{-h}\Big[\frac{1}{w}\frac{dp}{dx}(p_h)-\frac{1}{w}\frac{dp}{dx}(p)\Big].
 \end{align*}
By the induction assumption \eqref{eq:induct0} and $g_1(p_0)>0, \frac{dp}{dx}\leq 0,$ we have
 \begin{align}\label{I2}
   \frac{1}{-h}(I_2(p_h)-I_2(p))\leq C\quad \text{at}\quad p=p_0.
 \end{align}

 Summing up \eqref{I0}, \eqref{I1}, \eqref{I2}, \eqref{I3}, we conclude that at $p=p_0,$
 \beno
 \pa_x g_1-\sqrt{w}\pa_\psi^2 g_1\leq C_2 \psi^{-\frac{1}{2}}.
 \eeno
 This along with \eqref{gpsi2}  shows that for $x\geq\frac{5\epsilon}{8}$, we have at $p=p_0,$
\begin{align}
  \pa_x g-\sqrt{w}\pa_\psi^2  g\leq C \psi^{-\frac{1}{2}}-\gamma M\psi^{-\frac{1}{2}}.
\end{align}
Take $M$ large enough. Then we have $\pa_x g(p_0)-\sqrt{w}\pa_\psi^2  g(p_0)<0.$
However, if $p_0$ is the interior positive maximum point of $g$, we have
\begin{align*}
  \pa_x g(p_0)-\sqrt{w}\pa_\psi^2g(p_0)\geq0,
\end{align*}
which leads to a contradiction. Hence, for $M$ chosen as above and independent of $h,$ we have
$$\max_{\overline{D}}g\leq 0.$$

Similarly, we can show that $\min_{\overline{D}}g_2 \geq 0,$ where
$$ g_2=\left\{
\begin{aligned}
&\frac{\pa_x^{k-1}w(x-h,\psi)-\pa_x^{k-1}w(x,\psi)}{-h}\zeta-M\psi\ln \psi\,\,\,&\frac{5\epsilon}{8}\leq x\leq X,\psi\geq 0,\\
&-M\psi\ln \psi\,\,\,&0\leq x<\frac{5\epsilon}{8},\psi\geq0.
\end{aligned}
\right.
$$

Since $M$ is independent of $h$ and $h$ is arbitrary, we have
$$|\pa_x^k w|\leq -M\psi\ln\psi\quad \text{in} \quad(\frac{7}{8}\epsilon,X]\times(0,\delta_3].$$
Thanks to
\begin{align}\label{sqrtwx}
  2\sqrt{w}\pa_x^k\sqrt{w} +\sum_{m=1}^{k-1}C_{k}^m(\pa_x^m\sqrt{w}\pa_x^{k-m}\sqrt{w})=\pa_x^k(\sqrt{w}\sqrt{w})=\pa_x^k w,
\end{align}
we get by \eqref{eq:induct0}  that in $(\frac{7}{8}\epsilon,X]\times(0,\delta_3],$
$$|\sqrt{w}\pa_x^k\sqrt{w} |\leq -C\psi\ln\psi.$$
Then we have the desired result.
\end{proof}

\begin{lemma}\label{lem:s2}Assume that \eqref{eq:induct0} holds. Then
it holds that in $\big[\frac{15}{16}\epsilon,X\big]\times[0,\delta_3],$
\begin{align*}
|\pa_x^k w|\le C\psi,\quad |\pa_x^k\sqrt{ w}|\leq C\psi^{\frac{1}{2}}.
\end{align*}
\end{lemma}

\begin{proof}
Set
$$g=\pa_x^k w\zeta-A_1\psi+A_2\psi^{\frac{3}{2}-\beta},$$
where $\zeta(x)$ is a smooth cutoff function so that
 $0\leq\zeta(x)\leq 1$, $\zeta(x)\equiv 1$ in $[\frac{15}{16}\epsilon,X]$, $\zeta(x)\equiv 0$ in $\big[0,\frac{7\epsilon}{8}\big]$, and $A_1,A_2$ are large constants to be determined.

Take $\beta=\frac{1}{200}$ in Lemma \ref{lem:s1}. Then in $[\frac{7\epsilon}{8},X]\times[0,\delta_3],$
\begin{align}\label{811}
|\pa_x^k w|\le C\psi^{1-\beta},\quad |\pa_x^k\sqrt{ w}|\leq C\psi^{\frac{1}{2}-\beta}.
\end{align}

Set $D=(0,X]\times (0,\delta_3)$. Then $g$ is smooth in $D$ and continuous in $\overline{D}.$
 By \eqref{811}, we have $g(x,0)=0.$
Requiring $A_1\geq A_2,$ by the definition of $\zeta,$ we have $g(0,\psi)\leq 0.$
Taking $A_1$ large depending on $A_2$ so that for $x\in[0,A],$ $g(x,\delta_3)\leq 0.$

Next we show that by a proper choice of $A_1,A_2,$
 the maximum of $g$ cannot be achieved in the interior.
By \eqref{xd}, we have
\begin{align}\label{gpsi2-1}
  -\sqrt{w}\pa_\psi^2(-A_1\psi+A_2\psi^{\frac{3}{2}-\beta})\leq -\gamma A_2\psi^{-\beta},
\end{align} where $\gamma=(\frac{3}{2}-\beta)(\frac{1}{2}-\beta)\sqrt{m}>0.$
By \eqref{eq:sPrandtl-VM}, we have in $D,$
 \begin{align*}
  \pa_x \pa_x^{k} w-\sqrt{w}\pa_\psi^2 \pa_x^{k} w=&-2\frac{d^{k+1}p}{dx^{k+1}}+\sum_{m=0}^{k-1}C_k^m(\pa_x^{k-m}\sqrt{w})\pa_\psi^2 \pa_x^{m} w.
 \end{align*}
 By \eqref{eq:sPrandtl-VM}, we have
 \begin{align*}
    \pa_\psi^2 \pa_x^{m} w= \pa_x^{m}\pa_\psi^2 w=\pa_x^{m}\Big(\frac{\pa_xw}{\sqrt{w}}+\frac{2}{\sqrt{w}}\frac{dp}{dx}\Big).
 \end{align*}
 By \eqref{811} and \eqref{eq:induct0}, we have, for $x\geq\frac{7\epsilon}{8}$ and $0\leq j\leq k-1,$
$$|\pa_x^{j}w|\leq C\psi,\quad |\pa_x^{k}w|\leq C\psi^{1-\beta}.$$
Then for $0\leq m\leq k-1$ and $x\geq\frac{7\epsilon}{8},$
$$|\pa_\psi^2 \pa_x^{m} w|\leq C\psi^{-\frac{1}{2}}+C\psi^{\frac{1}{2}-\beta}\leq C\psi^{-\frac{1}{2}},$$
where we note $\beta<<\frac{1}{2}.$
On the other hand, by \eqref{811} and \eqref{eq:induct0}, we have for $0\leq m\leq k-1,$
$$|\pa_x^{k-m}\sqrt{w}|\leq C\psi^{\frac{1}{2}-\beta}.$$
Hence, we conclude that for $x\geq\frac{7\epsilon}{8},$
$$\pa_x \pa_x^k w-\sqrt{w}\pa_\psi^2 \pa_x^k w\leq C \psi^{-\beta},$$
from which and \eqref{gpsi2-1}, we deduce that at $p=p_0,$
\begin{align*}
  \pa_x g-\sqrt{w}\pa_\psi^2  g\leq C_2 \psi^{-\beta}-\gamma A_2\psi^{-\beta}.
\end{align*}

Now take $A_2$ large depending on $C_2$. Then we have $\pa_x g-\sqrt{w}\pa_\psi^2 g<0$ in $D$, which implies that the maximum of $g$ cannot be achieved in the interior. Hence,
$$\max_{\overline{D}}g\leq 0.$$
Similarly, we can prove that $\max_{\overline{D}}-\pa_x^k w\zeta-A_1\psi+A_2\psi^{\frac{3}{2}-\beta} \leq 0.$
Hence, we have
$$|\pa_x^k w|\leq A_1\psi-A_2\psi^{\frac{3}{2}-\beta} \leq  A_1\psi\quad \text{in}\quad \big[\frac{15}{16}\epsilon,X\big]\times[0,\delta_3].$$
Moreover, by \eqref{sqrtwx} and \eqref{eq:induct0}, we have that in $\big[\frac{15}{16}\epsilon,X\big]\times[0,\delta_3],$
$$|\pa_x^k\sqrt{w} |\leq C\psi^{\frac{1}{2}}.$$

This proves our desired result.
\end{proof}

\begin{lemma}\label{lem:s3}
Assume that \eqref{eq:induct0} holds. Then
it holds that in $[\epsilon,X]\times[0,\delta_4],$
\begin{align*}
|\pa_\psi\pa_x^{k} w|\leq C,\quad |\pa_\psi^2\pa_x^{k} w|\le C\psi^{-1},\quad |\pa_x^{k+1}w|\leq C\psi^{-\frac{1}{2}}.
\end{align*}

\end{lemma}
\begin{proof}
By Lemma \ref{lem:s2} and \eqref{eq:induct0}, it holds that in $\big[\frac{15}{16}\epsilon,X\big]\times[0,\delta_3],$
\begin{align}\label{235}
|\pa_x^j w|\le C\psi,\quad |\pa_x^j\sqrt{ w}|\leq C\psi^{\frac{1}{2}},\,\,0\leq j\leq k.
\end{align}

For $(x_3,\psi_3)\in[\epsilon,X]\times (0,\delta_0]$ with $\delta_0=\min\big\{\frac{2}{3}\delta_3,\frac{\epsilon}{16}\big\}$, we denote
$$Q=\big\{x_3-\psi_3^{\frac{3}{2}}\leq x\leq x_3\big\}\times \Big\{\frac{1}{2}\psi_3\leq \psi\leq\frac{3}{2}\psi_3\Big\}.$$

By straight calculations, we obtain
 \begin{align*}
  \pa_x \pa_x^{k} w-\sqrt{w}\pa_\psi^2 \pa_x^{k} w=&-2\frac{d^{k+1}p}{dx^{k+1}}+\pa_x^{k}\sqrt{w}\pa_\psi^2  w+\sum_{m=1}^{k-2}C_k^m(\pa_x^{k-m}\sqrt{w})\pa_\psi^2 \pa_x^{m} w
  \\&+C_k^{k-1}\frac{\pa_x w}{2\sqrt{w}}\pa_\psi^2 \pa_x^{k-1} w.
 \end{align*}
 Since
  \begin{align*}\begin{split}
    \pa_\psi^2 \pa_x^{m} w&= \pa_x^{m}\pa_\psi^2w=\pa_x^{m}\Big(\frac{\pa_xw}{\sqrt{w}}+\frac{2}{\sqrt{w}}\frac{dp}{dx}\Big) \\& =\frac{\pa_x^{m+1}w}{\sqrt{w}}+\sum_{l=1}^{m}C_m^l\pa_x^{m-l+1}w\pa_x^l\frac{1}{\sqrt{w}}+\pa_x^{m}\Big(\frac{2}{\sqrt{w}}\frac{dp}{dx}\Big)
\end{split}
\end{align*}
and
\begin{align*}
    \pa_x^{k}\sqrt{w}=  \pa_x^{k-1}\frac{\pa_x w}{2\sqrt{w}}=
    \frac{\pa_x^{k} w}{2\sqrt{w}}+
    \sum_{l=1}^{k-1}C_{k-1}^{l}\pa_x^{k-1-l+1} w\pa_x^{l}\frac{1}{2\sqrt{w}},
\end{align*}
 we obtain
 \begin{align*}
  \pa_x \pa_x^{k} w-\sqrt{w}\pa_\psi^2 \pa_x^{k} w=&-2\frac{d^{k+1}p}{dx^{k+1}}+\sum_{m=1}^{k-2}C_k^m(\pa_x^{k-m}\sqrt{w})\pa_\psi^2 \pa_x^{m} w\\&+ \frac{\pa_x^{k} w}{2\sqrt{w}}\pa_\psi^2  w+
    \sum_{l=1}^{k-1}C_{k-1}^{l}\pa_x^{k-1-l+1} w\Big(\pa_x^{l}\frac{1}{2\sqrt{w}}\Big)\pa_\psi^2  w
 \\& +C_k^{k-1}\frac{\pa_x w\pa_x^{k}w}{2w}+C_k^{k-1}\frac{\pa_x w}{2\sqrt{w}}\Big[\sum_{l=1}^{k-1}C_{k-1}^{l}\pa_x^{k-l}w\pa_x^l\frac{1}{\sqrt{w}}+\pa_x^{k-1}\Big(\frac{2}{\sqrt{w}}\frac{dp}{dx}\Big)\Big].
 \end{align*}

Make a transformation $T:Q\to \widetilde{Q}=[-1,0]_{\tilde{x}}\times[-\frac{1}{2},\frac{1}{2}]_{\tilde{\psi}}$ via
\begin{align*}
    x-x_3=\psi_3^{\frac{3}{2}}\tilde{x},\quad \psi-\psi_3=\psi_3\tilde{\psi}.
\end{align*}
Let $f=\pa_x^{k} w\psi_3^{-1},$ which satisfies
 \begin{align*}
  \pa_{\tilde{x}} f-\frac{\sqrt{w}}{\psi_3^{\frac{1}{2}}}\pa_{\tilde{\psi}}^2 f+cf=&-2\frac{d^{k+1}p}{dx^{k+1}}\psi_3^{\frac{1}{2}}+\psi_3^{\frac{1}{2}}\sum_{m=1}^{k-2}C_k^m(\pa_x^{k-m}\sqrt{w})\pa_\psi^2 \pa_x^{m} w\\&+\psi_3^{\frac{1}{2}}
    \sum_{l=1}^{k-1}C_{k-1}^{l}\pa_x^{k-l} w\Big(\pa_x^{l}\frac{1}{2\sqrt{w}}\Big)\pa_\psi^2  w
 \\& +\psi_3^{\frac{1}{2}}\frac{\pa_x w}{2\sqrt{w}}\Big[\sum_{l=1}^{k-1}C_{k-1}^{l}\pa_x^{k-l}w\pa_x^l\frac{1}{\sqrt{w}}+\pa_x^{k-1}\Big(\frac{2}{\sqrt{w}}\frac{dp}{dx}\Big)\Big]=F,
 \end{align*}
where $c=-\frac{1}{2\sqrt{w}}\pa_\psi^2  w\psi_3^{\frac{3}{2}}-\frac{\pa_x w}{2w}\psi_3^{\frac{3}{2}}.$

By \eqref{235}, we have
\begin{align*}
|f|\leq C \quad \text{in} \quad [-1,0]_{\tilde{x}}\times \big[-\frac{1}{2},\frac{1}{2}\big]_{\tilde{\psi}}.
\end{align*}
From the proof of Proposition \ref{prop:lower}, we know that for $\al\in (0,1)$,
\begin{align*}
0<c\leq \frac{\sqrt{w}}{\psi_3^{\frac{1}{2}}}\leq C,\quad  \Big|\frac{\sqrt{w}}{\psi_3^{\frac{1}{2}}}\Big|_ {\mathcal{C}^\alpha([-1,0]_{\tilde{x}}\times[-\frac{1}{2},\frac{1}{2}]_{\tilde{\psi}})}\leq C.
\end{align*}
Using the induction assumption \eqref{eq:induct0} and \eqref{235}, we can deduce that for $ j\leq k-1$ and $m\leq k-2$,
\begin{align*}
   |\nabla_{\tilde{x},\tilde{\psi}}\pa_x^j \sqrt{w}|\leq C\psi_3^{\frac{1}{2}},\quad \big|\nabla_{\tilde{x},\tilde{\psi}}\pa_x^j (\frac{1}{\sqrt{w}})\big| \leq C\psi_3^{-\frac{1}{2}},\quad |\nabla_{\tilde{x},\tilde{\psi}}\pa_\psi^2 \pa_x^{m} w | \leq C\psi_3^{-\frac{1}{2}},
\end{align*}
where in the third inequality we need to use the fact that
 \begin{align*}
    \pa_\psi(\pa_\psi^2 \pa_x^{m} w ) =&\frac{\pa_\psi\pa_x^{m+1}w}{\sqrt{w}}-\frac{\pa_\psi w\pa_x^{m+1}w}{2(\sqrt{w})^3}+\sum_{l=1}^{m}C_m^l\pa_x^{m-l+1}\pa_\psi w\pa_x^l\frac{1}{\sqrt{w}}\\&+\sum_{l=1}^{m}C_m^l\pa_x^{m-l+1} w\pa_x^l\frac{\pa_\psi w}{-2(\sqrt{w})^3}+\pa_x^{m}\Big(\frac{\pa_\psi w}{-(\sqrt{w})^3}\frac{dp}{dx}\Big).
\end{align*}
These along with \eqref{psipsialpha} and \eqref{psichusqrtw} imply hat
 $$|c|_{\mathcal{C}^\alpha(\widetilde{Q})}+|F|_{\mathcal{C}^\alpha(\widetilde{Q})}\leq C.$$
Then  standard interior a priori estimates yield that
\beno
|\pa_{\tilde{x}}f|_ {{L}^\infty([-\frac{1}{4},0]_{\tilde{x}}\times[-\frac{1}{8},\frac{1}{8}]_{\tilde{\psi}})}+|\pa_{\tilde{\psi}}f|_ {{L}^\infty([-\frac{1}{4},0]_{\tilde{x}}\times[-\frac{1}{8},\frac{1}{8}]_{\tilde{\psi}})}+|\pa_{\tilde{\psi}}^2f|_ {{L}^\infty([-\frac{1}{4},0]_{\tilde{x}}\times[-\frac{1}{8},\frac{1}{8}]_{\tilde{\psi}})}\leq C.
\eeno
Especially,  this implies
\beno
|\pa_x^{k+1} w(x_3,\psi_3)|\leq C\psi_3^{-\frac{1}{2}},\quad|\pa_{\psi}\pa_x^kw(x_3,\psi_3)|\leq C,\quad|\pa_{\psi}^2\pa_x^kw(x_3,\psi_3)|\leq C\psi_3^{-1}.
\eeno
Since $(x_3,\psi_3)$ is arbitrarily, we deduce the desired result.
\end{proof}

\begin{corol}\label{cor:infty}
Let $0<\epsilon<X$ and integer $k\ge 0$, there exists $\delta>0$  so that in $[\epsilon,X]\times[0,\delta],$
\begin{align}
|\pa_\psi^m\pa_x^{k} w|\le C\psi^{1-m}.\label{eq:w-infty}
\end{align}
\end{corol}

\begin{proof}
By Lemma \ref{lem:growth}, \eqref{xd}, Proposition \ref{prop:lower} and Proposition \ref{prop:high}, we know that  \eqref{eq:w-infty} holds for $m=0,1,2.$ Moreover, we have
\begin{align*}
   \Big|\pa_x^k\frac{1}{\sqrt{w}}\Big|\leq C\psi^{-\frac{1}{2}},\quad \Big|\pa_x^k\pa_\psi\frac{1}{\sqrt{w}}\Big|\leq C\psi^{-\frac{3}{2}},\quad \Big|\pa_x^k \pa_\psi^2\frac{1}{\sqrt{w}}\Big|\leq C\psi^{-\frac{5}{2}}.
\end{align*}
Let us inductively assume that \eqref{eq:w-infty} and
 \begin{align}\label{as1sqrtw}
    \Big|\pa_x^k\pa_\psi^m\frac{1}{\sqrt{w}}\Big|\leq C\psi^{-\frac{1}{2}-m}
\end{align}
hold for $0\leq m\leq l$ with $l\geq1.$
Next we show that for $m=l+1,$ \eqref{eq:w-infty} and \eqref{as1sqrtw} still hold. Thanks to
\begin{align*}
\pa_\psi^{l+1}\pa_x^{k} w=&\pa_\psi^{l-1}\pa_x^{k}\pa_\psi^2 w=\pa_x^{k}\pa_\psi^{l-1}\Big(\frac{\pa_xw}{\sqrt{w}}+\frac{2}{\sqrt{w}}\frac{dp}{dx}\Big)\\
=&\pa_x^{k}\Big(\sum_{a=0}^{l-1}C_{l-1}^a\pa_\psi^{l-1-a}\pa_xw\pa_\psi^{a}\frac{1}{\sqrt{w}}+2\frac{dp}{dx}\pa_\psi^{l-1}\frac{1}{\sqrt{w}}\Big)
\end{align*}
and the induction assumption, we have
\begin{align}\label{psimximprov}
 |\pa_\psi^{l+1}\pa_x^{k} w|\leq C\psi^{\frac{1}{2}-l}.
\end{align}
A direct calculation gives
\begin{align*}
   0=& \pa_x^k\pa_\psi^{l+1}\Big(\frac{1}{\sqrt{w}}\frac{1}{\sqrt{w}}w\Big)\\
   =&\pa_x^k\Big[2\sqrt{w}\pa_\psi^{l+1}\frac{1}{\sqrt{w}}
   +\sum_{a=1}^l\sum_{j=0}^{l+1-a}C_{l+1}^aC_{l+1-a}^j(\pa_\psi^{a}\frac{1}{\sqrt{w}})(\pa_\psi^{j}\frac{1}{\sqrt{w}})\pa_\psi^{l+1-j-a}w
   \\&\quad+\sum_{j=0}^{l}C_{l+1}^j\frac{1}{\sqrt{w}}(\pa_\psi^{j}\frac{1}{\sqrt{w}})\pa_\psi^{l+1-j}w\Big],
\end{align*}
from which and the induction assumption, we can deduce that
\begin{align*}
 \Big|\pa_x^k\pa_\psi^{l+1}\frac{1}{\sqrt{w}}\Big|\leq C\psi^{-\frac{3}{2}-l}.
\end{align*}
This along with \eqref{psimximprov} gives the desired result.
\end{proof}

\section{Global $C^\infty$ regularity}

In this section, we prove Theorem \ref{thm:reg}. As we remark below Theorem \ref{thm:reg}, it is enough to prove the regularity of the solution on the domain $[\epsilon, X]\times[0, \delta]$ for some small $\delta>0$.

\begin{proof}
To make the notation clear, we write the Von Mises transformation as follows
$$(\tilde{x},\psi)=\Big(x,\int_{0}^{y}u\,dy\Big).$$
A direct calculation(or See P25 in \cite{Olei}) shows that
\begin{align*}
&\pa_y=\sqrt{w}\pa_\psi,\quad \pa_x=\pa_{\tilde{x}}+\pa_x\psi(x,y)\pa_\psi,\\
& \pa_x\psi=\frac{1}{2}\sqrt{w}\int_{0}^\psi w^{-\frac{3}{2}}\pa_{\tilde{x}}w\,d\psi.
\end{align*}
By   \eqref{xd} and Proposition \ref{prop:high}, we have
\ben
|\pa_x\psi|\le C\psi. \label{eq:psi}
\een
Thanks to  $\pa_y=\sqrt{w}\pa_\psi$, we find that
\begin{align*}
    \pa_x^k2\pa_y u=&(\pa_{\tilde{x}}+\pa_x\psi\pa_\psi)^k\pa_{\psi} w,\\
    \pa_x^k2\pa_y^2 u=&(\pa_{\tilde{x}}+\pa_x\psi\pa_\psi)^k\Big(\pa_{\tilde{x}} w+2\frac{dp}{dx}\Big)\\=&(\pa_{\tilde{x}}+\pa_x\psi\pa_\psi)^k(\pa_{\tilde{x}} w)+2\frac{d^{k+1}p}{d x^{k+1}}.
\end{align*}
Using \eqref{eq:psi} and Corollary \ref{cor:infty}, we can deduce that
\begin{align}\label{y2xku}
|\pa_x^k\pa_y u|+|\pa_x^k\pa_y^2 u|\leq C.
\end{align}
Assume that for $m\geq 1$ and any integer $k$,
\begin{align}\label{xlyju}
|\pa_x^k\pa_y^{j}u|\leq C,\quad j\leq m.
\end{align}
Using the fact that
\begin{align*}
   \pa_x^k\pa_y^{m+1}u=&\pa_x^k\pa_y^{m-1}\pa_{y}^2u=\pa_x^k\pa_y^{m-1}\Big(u\pa_x u +\pa_{y}u\int_0^y-\pa_x u\,dy\Big)\\=&\pa_x^k\Big(\sum_{a=0}^{m-1}C_{m-1}^a\pa_y^{m-1-a}u\pa_y^{a}\pa_x u-
   \sum_{a=0}^{m-2}C_{m-1}^{a+1}\pa_y^{m-1-a}u\pa_y^{a}\pa_x u +\pa_{y}^mu\int_0^y-\pa_x u\,dy\Big),
\end{align*}
we can deduce from  \eqref{y2xku} and \eqref{xlyju} that
$$|\pa_x^k\pa_y^{j}u|\leq C, \quad j\leq m+1.$$

This proves the theorem.
\end{proof}

\section*{Acknowledgments}

The authors thank Professor Yan Guo for helpful discussions. Y. Wang is partially supported by Beijing Advanced Innovation Center for Imaging Theory and Technology and key research project of the Academy for Multidisciplinary Studies, Capital Normal University.

 \end{document}